\documentclass[12pt,reqno]{amsart}
\usepackage{amsfonts,amsmath,amssymb,times}

\usepackage[dvips]{graphicx}

\allowdisplaybreaks

\newtheorem{theorem}{Theorem}

\newtheorem{corollary}{Corollary}
\theoremstyle{definition}
\newtheorem{remark}{Remark}
\newtheorem{example}{Example}
\newtheorem{definition}{Definition}
\newtheorem*{ack}{Acknowledgement}
\newtheorem{urn}{Urn}

\def\P{{\mathbb {P}}}
\def\E{{\mathbb {E}}}

\newcommand{\Seq}{\textsc{SEQ}}

\newcommand{\N}{\ensuremath{\mathbb{N}}}

\newcommand{\refT}[1]{Theorem~\ref{#1}}

\newcommand{\refD}[1]{Definition~\ref{#1}}
\newcommand{\refR}[1]{Remark~\ref{#1}}
\newcommand{\refS}[1]{Section~\ref{#1}}

\newcommand{\refE}[1]{Example~\ref{#1}}

\newcommand{\refUA}{Urn~\ref{UA}}
\newcommand{\refUB}{Urn~\ref{UB}}

\newcommand{\refUCx}[1]{Urn~\ref{UC}${}_{#1}$}
\newcommand{\refUCm}{\refUCx m}

\newcommand\ntoo{\ensuremath{{n\to\infty}}}
\newcommand\cQ{{\mathcal{Q}}}

\newcommand\bcQ{\overline{\mathcal{Q}}}

\newcommand\cA{\mathcal A}
\newcommand\cB{\mathcal B}

\newcommand\cP{\mathcal P}
\newcommand\cS{\mathcal S}
\newcommand\cT{\mathcal T}

\newcommand\tV{\widetilde V}
\newcommand\tK{\widetilde K}
\newcommand\tcK{\widetilde {\mathcal{K}}}
\newcommand\tKn[1]{\widetilde K_{n,#1}}
\newcommand\tcKn[1]{\widetilde {\mathcal{K}}_{n,#1}}
\newcommand\Kn[1]{K_{n,#1}}
\newcommand\bbN{\mathbb N}

\newcommand\ga{\alpha}
\newcommand\gb{\beta}
\newcommand\gd{\delta}
\newcommand\gf{\varphi}
\newcommand\gl{\lambda}

\newcommand\Var{\operatorname{Var}}
\newcommand\Cov{\operatorname{Cov}}

\newcommand\PD{\operatorname{PD}}
\newcommand\Beta{\operatorname{Beta}}
 
\newcommand\eqd{\overset{\mathrm{d}}{=}}

\newcommand\set[1]{\ensuremath{\{#1\}}}
\newcommand\xpar[1]{(#1)}
\newcommand\bigpar[1]{\bigl(#1\bigr)}
\newcommand\Bigpar[1]{\Bigl(#1\Bigr)}

\newcommand\lrpar[1]{\left(#1\right)}

\newcommand\Sp{Stirling permutation}
\newcommand\gSp{generalized \Sp}
\newcommand\kSp{$k$-\Sp}
\newcommand\kbSp{$k$-bundled \Sp}
\newcommand\xL{\tilde L}
\newcommand\gpu{generalized P\'olya urn}
\newcommand\zzeta{\tilde\zeta}
\newcommand\W{W}
\newcommand\B{B}
\newcommand{\alaw}{\ensuremath{\xrightarrow{(a.s.)}}}
\newcommand{\dto}{\ensuremath{\xrightarrow{(d)}}}

\title[Stirling permutations, increasing trees and urn models]
{Generalized Stirling permutations, families of increasing trees and urn models}

\author[S.~Janson]{Svante Janson}
\address{Svante Janson\\
Department of Mathematics, Uppsala University, PO Box 480, SE-751 06
Uppsala, Sweden}%
\email{svante.janson@math.uu.se}
\urladdr{http://www.math.uu.se/$\sim$svante/}

\author[M.~Kuba]{Markus Kuba}
\address{Markus Kuba\\
Institut f{\"u}r Diskrete Mathematik und Geometrie\\
Technische Universit\"at Wien\\
Wiedner Hauptstr. 8-10/104\\
1040 Wien, Austria} %
\email{kuba@dmg.tuwien.ac.at}

\author[A.~Panholzer]{Alois Panholzer}
\address{Alois Panholzer\\
Institut f{\"u}r Diskrete Mathematik und Geometrie\\
Technische Universit\"at Wien\\
Wiedner Hauptstr. 8-10/104\\
1040 Wien, Austria} \email{Alois.Panholzer@tuwien.ac.at}
\thanks{The second and third author were supported by the Austrian Science Foundation FWF, grant S9608.}

\date{May 26, 2008}  

\begin{document}

\begin{abstract}
Bona \cite{Bona2007} studied the distribution of ascents, plateaux and descents in the class
of Stirling permutations, introduced by Gessel and Stanley \cite{GessStan1978}.
Recently, Janson \cite{Jan2008} showed the connection between Stirling permutations and plane recursive trees and proved
a joint normal law for the parameters considered by Bona. Here we will consider generalized Stirling permutations 
extending the earlier results of \cite{Bona2007}, \cite{Jan2008},
and relate them with certain families of generalized plane recursive trees, and
also $(k+1)$-ary increasing trees.  We also give two different bijections 
between certain families of increasing trees, which both give as a special case
a bijection between ternary increasing trees and plane recursive trees.
In order to describe the (asymptotic) behaviour of the parameters of interests, 
we study three (generalized) P\'olya urn models using various methods.\end{abstract}

\maketitle


\section{Introduction\label{STIRPsec1}}

Stirling permutations were defined by Gessel and Stanley \cite{GessStan1978}.
A Stirling permutation is a permutation of the multiset
$\{1, 1, 2, 2, \dots , n, n\}$ such that for each $i$, $1\le i \le n$, the elements occuring between
the two occurences of $i$ are larger than $i$. The name of these combinatorial objects is due to relations
with the Stirling numbers, see \cite{GessStan1978} for details.

Let $\sigma = a_1a_2 \dotsm a_{2n}$ be a Stirling permutation. Let the
index $i$ (or the gap $(i,i+1)$)
be
called an ascent of $\sigma$ if $i = 0$ or $a_i < a_{i+1}$, let $i$ be called a descent of $\sigma$
if $i = 2n$ or $a_i > a_{i+1}$, and let $i$ be called a plateau of
$\sigma$ if $a_i = a_{i+1}$.
(It is convenient to define $a_0=a_{2n+1}=0$; this takes care of the
special cases $i=0$ and $i=2n$.) Note that $i$ runs from $0$ to $2n$,
so the total number of ascents, descents and plateaux is $2n+1$.
Let $\mathcal{Q}_n$ denote the set of Stirling permutation of  
$\{1,1, 2, 2,\dots , n, n\}$; we say that these have order $n$.
Bona \cite{Bona2007} showed that the parameters numbers of ascents,
descents and 
plateaux are equidistributed on  
$\mathcal{Q}_n$. Moreover, he showed a central limit theorem for the
three parameters. 

A rooted tree of order $n$ with the vertices labelled $1,2,\dots,n$,
is an increasing tree if the node 
labelled 1 is distinguished as the root, and for each $2 \le k \le n$,
the labels of
the nodes in the unique path from the root to the node labelled $k$ form an
increasing sequence.
We will consider several families of increasing trees. The first one
is the family of \emph{increasing plane trees}, usually called 
\emph{plane recursive trees}, where the children of a node are ordered
(from left to right, say).
Note that plane recursive trees also appear in literature under the
names plane-oriented recursive trees, heap-ordered trees, and
sometimes also as scale-free trees. 
Further families will be defined later.


Let $\mathcal{T}_n$ denote the set of plane recursive trees with $n$
vertices. It was shown by Janson \cite{Jan2008} that plane recursive
trees on $n+1$ vertices are in bijection with Stirling permutations
on $\{1, 1, 2, 2, \dots , n, n\}$, $\mathcal{T}_{n+1} \cong
\mathcal{Q}_n$. Moreover, using this bijective correspondence, he
showed that the number of descents in the \Sp{} corresponds to the
number of leaves in the associated plane recursive tree.
Furthermore, using an urn model and general theorems, see
\cite{Jan2004} and also \cite{Jan2005}, Janson showed the joint
normality of the parameters ascent, descent and plateau.
The purpose of this work is to extend this connection between
Stirling permutations and plane recursive trees in
Janson \cite{Jan2008}, to generalized Stirling permutations. 
In particular, we give a bijection between
Stirling permutations on $\{1^k, 2^{k}, \dots, n^{k}\}$, where here and throughout this work $1^l:=\underbrace{1,\dots ,1}_{l}$, with $l\ge 1$,
which we call $k$-Stirling permutations, and $(k+1)$-ary increasing trees; moreover we can also relate $k$-Stirling permutations with
a certain family of plane recursive trees, namely $k$-plane recursive trees.
Concerning Stirling permutations of the multiset $\{1^k, 2^{k+2}, \dots,
n^{k+2}\}$, which we call  $k$-bundled Stirling
permutations, 
we obtain a bijection with certain generalized plane recursive trees,
namely $k$-bundled increasing trees. We also give two different bijections 
between certain families of increasing trees, which both give as a special case
a bijection between ternary increasing trees and plane oriented
increasing trees. 
Moreover, we will use several different
methods, combinatorial and probabilistic, to derive several
results in this direction. More precisely, in order 
to describe the (asymptotic) behaviour of the parameters of interests, 
we study three (generalized) P\'olya urn models.

The parameter $k$ is fixed throughout the paper, and often omitted
from the notation.
All unspecified limits are as \ntoo.
In the results with a.s.\ convergence, we assume that the random
\kSp{} 
grows in the natural way by random addition of new labels; in the
other results, this does not matter.

\section{Preliminaries}\label{Sprel}

\subsection{Generalized Stirling permutations}
A straightforward generalization of Stirling permutations on the
multiset $\{1, 1, 2, 2, \dots , n, n\}$ is to consider
permutations of
a more general multiset 
$\{1^{k_1}, 2^{k_2}, \dots, n^{k_n}\}$, with $k_i\in\N$ for $1\le i\le n$.
We call a permutation of the multiset $\{1^{k_1}, 2^{k_2},\dots, n^{k_n}\}$
a \emph{generalized Stirling permutation}, if for each
$i$, $1\le i \le n$, the elements occurring between two occurrences of
$i$ are at least $i$. (In other words, the elements occurring between two
consecutive occurrences of $i$ are larger than $i$.)
Such permutations have already previously been considered by Brenti
\cite{Brenti1989}, \cite{Brenti1998}. 
The number of generalized \Sp{s} of $\{1^{k_1}, 2^{k_2},\dots, n^{k_n}\}$
is
\begin{equation}\label{gsp}
  \prod_{i=1}^{n-1}(\ell_i+1)\quad\text{with } \ell_i=\sum_{j=1}^ik_j;
\end{equation}
this is easy to see by induction, since the $k_n$ copies of $n$ have
to form a substring, and this substring can be inserted in $\ell_{n-1}+1$
positions (viz., anywhere, including first or last)
in any \gSp{} of $\{1^{k_1}, 2^{k_2},\dots, (n-1)^{k_{n-1}}\}$.

We will consider two cases and give them special names:
a \emph{$k$-Stirling permutation} of order $n$ is a generalized
Stirling permutation of the multiset $\{1^k, 2^k,\dots, n^k\}$,
and
a \emph{$k$-bundled Stirling permutation} is
a generalized Stirling permutation of the multiset
$\{1^k, 2^{k+2}, \dots, n^{k+2}\}$. Here $k\ge1$, but note that
$1$-Stirling permutations are just ordinary permutations so we will
usually consider $k$-Stirling permutations for $k\ge2$ only; the case
$k=2$ yields the ordinary Stirling permutations defined by Gessel and Stanley
\cite{GessStan1978}. 

What we call \kSp{s}
was suggested by Gessel and Stanley \cite{GessStan1978} 
and has been studied by Park \cite{Park1994b,Park1994a,Park1994c} under
the name $k$-multipermutations.

In the following, let
$\mathcal{Q}_n=\mathcal{Q}_n(k)$ denote the set of $k$-Stirling
permutations of order $n$ and
let $Q_n=Q_n(k)$ denote the number $|\cQ_n(k)|$ of them.
By \eqref{gsp},
\begin{equation}\label{qn}
  Q_n(k)
=|\cQ_n(k)|
= \prod_{i=1}^{n-1}(ki+1)
=k^n\frac{\Gamma(n+1/k)}{\Gamma(1/k)}.
\end{equation}
For $k=2$ this number is just $Q_n(2)=(2n-1)!!$. 
In the case 
$k=3$, we have for example one permutation of order $1$: $111$; four
permutations of order $2$: $111222$, $112221$, $122211$, $222111$; etc. 

Similarly, let $\overline{\mathcal{Q}}_n=\overline{\mathcal{Q}}_n(k)$
denote the set of $k$-bundled Stirling permutations of order $n$ and
let $\overline{Q}_n=\overline{Q}_n(k)$ 
denote the number of them.
We have, by \eqref{gsp},
\begin{equation}
  \label{bqn}
\overline{Q}_n
=|\bcQ_n(k)|
=\prod_{i=1}^{n-1}(i(k+2)-1)
=(k+2)^{n-1}\frac{\Gamma(n-1/(k+2))}{\Gamma(1-1/(k+2))}.
\end{equation}

We define \emph{ascents}, \emph{descents} and \emph{plateaux} of a \gSp{} 
$\sigma = a_1a_2\dotsm a_{\ell}$ 
of $\{1^{k_1}, 2^{k_2},\dots, n^{k_n}\}$
(where the length $\ell=\sum_1^n k_i$)
as before: we let $a_0=a_{\ell+1}=0$ and say that an 
index $i$, with $0\le i \le \ell$, is an ascent, descent or plateau
if $a_i<a_{i+1}$, $a_i>a_{i+1}$ or $a_i=a_{i+1}$, respectively.
Note that the total number of them is $\ell+1$.

We introduce a natural refinement of ascents, descents and plateaux,
namely \emph{$j$-ascents}, \emph{$j$-descents}, and \emph{$j$-plateaux}.
An index $i$, with $1\le i\le \ell$ 
is called a $j$-ascent, 
if $i$ is
an ascent and there are exactly $j-1$ indices $i'<i$ such that
$a_{i'}=a_{i}$; \i.e., $a_i$ is the $j$th occurrence of the symbol $a_i$,
and similarly for plateaux. 
For a descent $i$, $a_i$ is always the last occurence of that symbol
(just as for an ascent, $a_{i+1}$ is the first of its kind), and we define
a $j$-descent as a descent $i<\ell$ such that $a_{i+1}$ is the
the $j$th occurrence of that symbol.
(Note that we choose not to allow $i=0$ or $i=\ell$ in these definitions.)

Thus, for a \gSp{} 
of $\{1^{k_1}, 2^{k_2},\dots, n^{k_n}\}$, the possible values of $j$
ranges from 1 to $\max_i k_i$ for $j$-ascents and $j$-descents,
and from 1 to $\max_i k_i-1$ for $j$-plateaux.
In particular, for \kSp{s}, $1\le j\le k$ for 
$j$-ascents and $j$-descents, and 
$1\le j\le k-1$ for $j$-plateaux.
Note also that if we reflect a \kSp, we get a new \kSp, and
$j$-ascents in one of them correspond to $(k+1-j)$-descents in the other.

\begin{example}
Consider the 3-Stirling permutation $\sigma=112233321$: 
Index 1 is a 1-plateau, index 2 is a $2$-ascent, 
index 3 is a 1-plateau, index 4 is a $2$-ascent, index 5 is a 1-plateau,
index $6$ is a 2-plateau, index 7 is a $3$-descent, and index 8 is a
3-descent.
(Indices 0 and 9 are not classified in this way.)
\end{example}

We are interested in
the (joint) distributions of the random variables $X_{n,j}$, $Y_{n,j}$
and $Z_{n,j}$, defined as the numbers of 
$j$-ascents, $j$-descents and $j$-plateaux, respectively, in a random
$k$-Stirling permutation (chosen uniformly in $\mathcal{Q}_n(k)$). 
Note that these trivially are 0 unless 
$1\le j\le k$ for $X_{n,j}$ and $Y_{n,j}$,
and $1\le j\le k-1$ for $Z_{n,j}$, and that
\begin{equation*}
\sum_{j=1}^{k}(X_{n,j} + Y_{n,j}) + \sum_{j=1}^{k-1}Z_{n,j} = kn - 1
\end{equation*}

We further let $X_n$, $Y_n$ and $Z_n$ denote the total numbers of
ascents, descents and plateaux, respectively. Note that, recalling
the special definitions at the endpoints,
\begin{align}
  X_n&=\sum_{j=1}^kX_{n,j}+1, \label{xyzx}
\\
  Y_n&=\sum_{j=1}^kY_{n,j}+1, \label{xyzy}
\\
  Z_n&=\sum_{j=1}^kZ_{n,j}. \label{xyzz}
\end{align}

It is easy to see that a $j$-ascent with $j<k$ corresponds to a later
$(j+1)$-descent, and conversely, so 
\begin{equation}
X_{n,j}=Y_{n,j+1},
\qquad1\le j \le k-1, 
\end{equation}
see also \refT{STIRPthe1}.
However, there is no corresponding relation for $k$-ascents, of for
1-descents, and the total numbers of ascents and descents are
typically different, even in the case $k=2$.
Further, since only the last copy of a label can be a descent,
\begin{equation}\label{xz}
X_{n,j}+Z_{n,j}=n,
\qquad1\le j \le k-1,
\end{equation}
and, similarly or by \eqref{xz},
\begin{equation}
Y_{n,j}+Z_{n,j-1}=n,
\qquad2\le j \le k. 
\end{equation}

Moreover, we are also interested in the distribution of the number 
of \emph{blocks} in a random $k$-Stirling permutation of order $n$. 
A block in a \gSp{} $\sigma=a_1\dotsm a_\ell$ is a substring
$a_p\dotsm a_q$ with $a_p=a_q$
that is maximal, i.e.~not contained in any larger
such substring. 
There is obviously at most one block for every $j=1,\dots,n$,
extending from the first occurrence of $j$ to the last; we say that
$j$ forms a block when this substring really is a block, i.e.~when it
is not contained in a string $i\dotsm i$ for some $i<j$.
In particular, in a \kSp, 
$j$ forms a block if for any $i$ with $1\le i \le j-1$, there do not exist
indices $m_0,\dots m_{k+1}$ , with $1\le m_0<\dots<m_{k+1}\le kn$, such that 
$\sigma_{m_0}=\sigma_{m_{k+1}}=i$ and
$\sigma_{m_1}=\dots=\sigma_{m_k}=j$.
It is easily seen by induction that any \gSp{} has a unique
decomposition as a sequence of its blocks.
Note that if we add a string $(n+1)^{k_{n+1}}$ to a \gSp, this string
will either be swallowed by one of the existing blocks, or
form a block on its own; the latter happens when it is added first,
last, or in a gap
between two blocks.

\begin{example}
The 3-Stirling permutation $\sigma=112233321445554666$, has 
block decomposition $[112233321][445554][666]$.
\end{example}

\smallskip

One may also consider the similar problems for $k$-bundled \Sp{s};
similarly defining random variables
$\overline{X}_{n,j}$, $\overline{Y}_{n,j}$ and $\overline{Z}_{n,j}$.
However, for most results we restrict ourselves to
$k$-Stirling permutations. 

\subsection{Generalized plane recursive trees and $d$-ary increasing trees}
In order to relate the $k$-Stirling permutations to families of
increasing trees 
we use a general setting based on earlier considerations of
Bergeron et al.~\cite{BerFlaSal1992} and Panholzer and Prodinger
\cite{PanPro2005+}. 

For a given degree-weight sequence $(\varphi_{k})_{k \ge 0}$, the corresponding degree-weight generating function $\varphi(t)$ is defined by $\varphi(t) := \sum_{k \ge 0} \varphi_{k} t^{k}$.
The simple family of increasing trees $\mathcal{T}$ associated with a degree-weight generating function $\varphi(t)$, can be described
by the formal recursive equation \begin{equation}
   \label{eqnz0}
   \mathcal{T} = \bigcirc\hspace*{-0.75em}\text{\small{$1$}}\hspace*{0.4em}
   \times \Big(\varphi_{0} \cdot \{\epsilon\} \; \dot{\cup} \;
   \varphi_{1} \cdot \mathcal{T} \; \dot{\cup} \; \varphi_{2} \cdot
   \mathcal{T} \ast \mathcal{T} \; \dot{\cup} \; \varphi_{3} \cdot
   \mathcal{T} \ast \mathcal{T} \ast \mathcal{T} \; \dot{\cup} \; \cdots \Big)
   = \bigcirc\hspace*{-0.75em}\text{\small{$1$}}\hspace*{0.3em} \times \varphi(\mathcal{T}),
\end{equation}
where $\bigcirc\hspace*{-0.75em}\text{\small{$1$}}\hspace*{0.4em}$ denotes the node
labelled by $1$, $\times$ the cartesian product, $\dot{\cup}$ the disjoint union, $\ast$ the partition product for labelled
objects, and $\varphi(\mathcal{T})$ the substituted structure (see
e.~g., the books 
\cite{VitFla1990}, \cite{FlaSed2008}).
This means that the elements of $\cT$ are increasing plane trees, and
that a tree with (out-)degrees $d_1,\dots,d_n$ is given weight
$\prod_1^n\gf_{d_i}$. 
By a random tree of order $n$ from the family $\cT$, we mean a tree of
order $n$ chosen randomly with probabilities proportional to the weights.

Let $T_n$ be the total weight of all such trees of order $n$.
It follows from \eqref{eqnz0} that the exponential generating function
$T(z) := \sum_{n \ge 1} T_{n} \frac{z^{n}}{n!}$ of the total weights
satisfies the autonomous first order differential equation
\begin{equation}
   \label{eqnz1}
   T'(z) = \varphi\big(T(z)\big), \quad T(0)=0.
\end{equation}

The families that we will consider have degree-weights of one of the
two following forms, studied by Panholzer and Prodinger \cite{PanPro2005+}:
\begin{equation}\label{pp}
\varphi(t) = \begin{cases}\frac{\varphi_{0}}
      {(1 + \frac{c_{2} t}{\varphi_{0}})^{-\frac{c_{1}}{c_{2}}-1}},
      \enspace \text{for} \enspace \varphi_{0} > 0, \; 0 < -c_{2} <
       c_{1},
\quad\text{\em generalized plane recursive trees,}\\
       \varphi_{0}
      \Big(1 + \frac{c_{2} t}{\varphi_{0}}\Big)^{d},
      \enspace \text{for} \enspace \varphi_{0},c_2 > 0,    \; d :=
	  \frac{c_{1}}{c_{2}}+1 \in \N\setminus\{1\},
\quad \text{\em $d$-ary increasing trees}.
      \end{cases}
\end{equation}
Consequently, by solving \eqref{eqnz1}, we obtain exponential generating function $T(z)$
\begin{equation}\label{t}
T(z)=
     \begin{cases}
   \frac{\varphi_{0}}{c_{2}} \Big(\frac{1}{(1-c_{1} z)^{\frac{c_{2}}{c_{1}}}} -1 \Big),\enspace\text{generalized plane recursive trees},\\
    \frac{\varphi_{0}}{c_{2}} \Big(\frac{1}{(1-(d-1) c_{2} z)^{\frac{1}{d-1}}} - 1 \Big),\enspace d\text{-ary increasing trees},
    \end{cases}
\end{equation}
 and the total weights $T_n$,
\begin{equation}\label{tn}
T_{n} = \varphi_{0} c_{1}^{n-1} (n-1)! \binom{n-1+\frac{c_{2}}{c_{1}}}{n-1}.
\end{equation}
Note that changing $\gf_k$ to $ab^k\gf_k$ for some positive constants
$a$ and $b$ will affect the weights of all trees of a given order $n$
by the same factor $a^nb^{n-1}$, which does not affect the
distribution of a random tree from the family. Hence, when considering
random trees from these two classes, $\gf_0$ is irrelevant and $c_1$
and $c_2$ are relevant only through the ratio $c_1/c_2$. (We may thus,
if we like, normalize $\gf_0=1$ and either $c_1$ or $|c_2|$, but not both.)

As shown by Panholzer and Prodinger \cite{PanPro2005+}, random trees
in the two classes of families given in \eqref{pp}
can be grown as an evolution process in the
following way. 
The process, evolving in discrete time, starts with the root labelled by $1$.
At step $i+1$ the node with label $i+1$ is attached to any previous
node $v$ (with out-degree $d(v)$) of the already grown tree of order
$i$ with probabilities $p(v)$ 
given by
\begin{equation*}
   p(v)=
   \begin{cases}
   \frac{d(v)+\alpha}{(\alpha+1)i-1} \quad \text{with} \enspace \alpha := -1-\frac{c_{1}}{c_{2}} > 0,\enspace\text{generalized plane recursive trees},\\
   \frac{d-d(v)}{(d-1)i+1},\enspace d\text{-ary increasing trees}.
   \end{cases}
\end{equation*}
Moreover, Panholzer and Prodinger \cite{PanPro2005+} showed that there
are only three classes of simple families that can be grown in this
way (for suitable $p(v)$): the two classes given in \eqref{pp} and the
recursive trees given by $\gf(t)=\gf_0e^{c_1t/\gf_0}$ with $\gf_0,c_1>0$
(which can be regarded as a limiting case of any of the two classes
above, letting $c_2\to0$.)

\begin{example}\label{Eplane}
\emph{Plane recursive trees} are 
   plane increasing trees such that all node degrees are allowed,
with all trees having weight 1. Thus $\gf_k=1$ and
   the degree-weight generating function is $\varphi(t) =\frac{1}{1-t}$,
which is of the form in \eqref{pp}
   with $\varphi_0=1$, $c_1=2$ and $c_2=-1$.
   We have
   \begin{equation*}
      T(z) = 1-\sqrt{1-2z}, \quad \text{and} \quad
      T_{n} = 1 \cdot 3 \cdot 5 \cdots (2n-3)
      = (2n-3)!!, \enspace \text{for} \; n \ge 1.
   \end{equation*}
   Furthermore, $\alpha=-1-\frac{c_1}{c_2}=1$, and consequently, the
   probability attaching to node $v$ at step $i+1$ is given by
   $p(v)=\frac{d(v)+1}{2i-1}$. 
\end{example}

\begin{example}\label{Ed}
For an integer $d\ge2$, \emph{$d$-ary increasing trees} are
increasing trees where each node has $d$ (labelled) positions for children.
Thus, only outdegrees $0,\dots,d$ are allowed; moreover, for a node
with $k$ children in given order, there is thus $\binom dk$ ways to
attach them. Hence, this family is given by vertex weights
$\gf_k=\binom dk$ and thus 
the degree-weight generating function $\varphi(t) ={1+t}^d$,
which is of the form in \eqref{pp}
   with $\varphi_0=1$, $c_1=d-1$ and $c_2=1$.
By \eqref{t},
   \begin{equation*}
      T(z) = \bigpar{1-(d-1)z}^{-1/(d-1)}-1.
   \end{equation*}
\end{example}

\section{Increasing trees associated to generalized Stirling permutations}
\label{Sinc}

\subsection{$(k+1)$-ary increasing trees, $k$-plane recursive trees
   and $k$-Stirling permutations} 
\label{SSinck}

Recall from \refE{Ed} that, for $k\ge1$,
the degree-weight generating function of $(k+1)$-ary increasing trees
is given by 
$\varphi(t)=(1+t)^{k+1}$, i.e.~$\varphi_0=1$, $c_1=k$ and $c_2=1$.
Consequently, the generating function $T(z)$ and the numbers $T_n$ of
$(k+1)$-ary trees of order $n$ are given by
\begin{equation*}
T(z)=\frac{1}{(1-kz)^{\frac1k}}-1,
\qquad T_{n}=\prod_{l=1}^{n}(k(l-1)+1), 
\quad
n\ge 1,
\end{equation*}
and the the probability of attaching to node $v$ at step $i+1$ is
given by $p(v)=\frac{k+1-d(v)}{ki +1 }$. 

Note that $T_n=Q_n$, the number of \kSp, which makes the following
theorem reasonable.

\begin{theorem}[Gessel]
\label{STIRPprop1T}
Let $k\ge1$.
The family $\mathcal{A}_n=\mathcal{A}_n(k+1)$ of $(k+1)$-ary
increasing trees of order $n$ is in a natural bijection
with $k$-Stirling permutations, $\mathcal{A}_n(k+1)\cong
\mathcal{Q}_n(k)$. 
\end{theorem}

\begin{remark}
The authors independently derived the result above, and later
discovered the work of Park~\cite{Park1994b},
in which Gessel's result was mentioned but the proof only sketched. The result of Gessel never appeared in print except this
mentioning in
Park~\cite{Park1994b}, to the best of the authors' knowledge. We will
give a detailed proof of the result above, which has interesting
consequences regarding the (refined) parameters ascents, descents and
plataeux, and also number of blocks, which we will state in
Theorem~\ref{STIRPthe1}. 
\end{remark}

\begin{remark}
For $k=1$ we obtain a bijection between $1$-Stirling permutations
(ordinary permutations) and binary increasing trees, which is very
well known. 
\end{remark}

\begin{proof}
We use a slightly modified bijection to the one given by Janson in
\cite{Jan2008} for Stirling permutation and plane recursive tree,
and use a depth-first walk. The depth-first walk of a rooted (plane)
tree starts at the root, goes first to the leftmost child of the
root, explores that branch (recursively, using the same rules),
returns to the root, and continues with the next child of the
root, until there are no more children left. We think of
$(k+1)$-ary increasing trees, where the empty places are represented
by ``exterior nodes''. Hence, at any time, any (interior) node has $k+1$
children, some of which may be exterior nodes. Between these $k$ edges
going out from a node labelled $v$, we place $k$ integers $v$. 
(Exterior nodes have no children and no labels.)
Now
we perform the depth-first walk and code the $(k+1)$-ary increasing
tree by the sequence of the labels visited as we go around the tree 
(one may think of actually going around the tree like drawing
the contour). 
In other words, we add label $v$ to the code the $k$ first times we
return to node $v$, but not the first time we arrive there or the last
time we return.
A $(k+1)$-ary increasing tree of order 1 is encoded
by $1^k$. A $(k+1)$-ary increasing tree of order $n$ is 
encoded by a string of $k\cdot n$ integers, where each of the labels
$1,\dots, n$ appears exactly $k$ times. In other words, the code is
a permutation of the multiset $\{1^k, 2^k,\dots, n^k\}$. Note that
for each $i$, $1\le i \le n$, the elements occurring between the two
occurrences of $i$ are larger than $i$, since we can only visit
nodes with higher labels.
Hence the code is a \kSp.
Moreover, adding a new node $n+1$ at one of the $kn+1$ free positions
(i.e., the positions occupied by exterior nodes) corresponds to
inserting the $k$-tuple $(n+1)^k$ 
in the code at one of $kn + 1$ gaps; note (e.g., by induction) that
there is a bijection between exterior nodes in the tree and gaps in
the code. 
This shows that the code determines the $(k+1)$-ary
increasing tree uniquely and that the coding is a bijection. 
See Figure~\ref{STIRPfig1} for an illustration.

The inverse, starting with a $k$-Stirling permutation $\sigma$ of order $n$
and constructing the corresponding $(k+1)$-ary increasing tree can be
described as follows. 
We proceed recursively starting at step one by
decomposing the permutation as $\sigma=\sigma_{1}1\sigma_2 1\dots
\sigma_{k} 1 \sigma_{k+1}$,  
where (after a proper relabelling) the $\sigma_i$'s are again $k$-Stirling permutations.
Now the smallest label in each $\sigma_i$ is attached to the root node
labelled 1. 
We recursively apply this procedure to each  $\sigma_i$ to obtain the
tree representation. 
\end{proof}

\begin{figure}[htb]
\centering
\includegraphics[angle=0,scale=0.7]{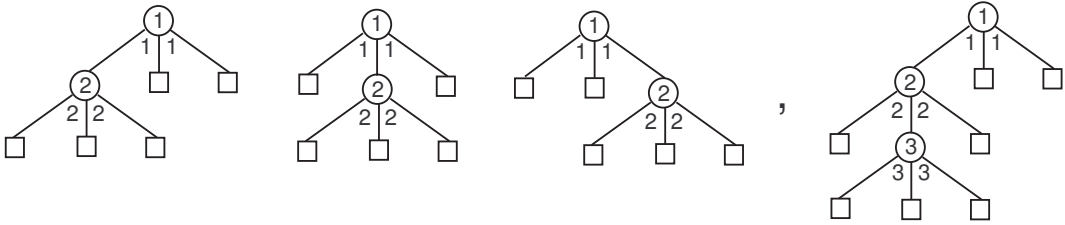}
\caption{The three ternary trees of order $2$ encoded by 2211, 1221
  and 1122; an order 3 ternary increasing tree encoded by the sequence
  233211.} \label{STIRPfig1} 
\end{figure}

Now we relate the distribution of $j$-ascents, $j$-descents and $j$-plateaux 
in $k$-Stirling permutations with certain parameters 
in $(k+1)$-ary increasing trees. In order to do so we introduce
two kinds of parameters. The parameter $D_{n,j}$, standing for ``$j$th
children'' , 
counts the number of nodes in a random $(k+1)$-ary increasing tree of order
$n$ that are the 
$j$th children of their respective parents, going from left to right,
with $1\le j \le k+1$.
Similarly, 
the parameter $L_{n,j}$, standing for ``leaves'' of type $j$, counts
the number of exterior nodes that
are $j$th children of their parents, $1\le j \le k+1$.
We thus have  
\begin{equation}
  \label{ld}
L_{n,j}=n-D_{n,j},
\qquad 1\le j\le k+1,
\end{equation}
and, counting the total numbers of interior and exterior children,
\begin{align}\label{dla}
\sum_{j=1}^{k+1}D_{n,j}=n-1,&&&
\sum_{j=1}^{k+1}L_{n,j}=kn+1.
\end{align}

Concerning the number of blocks in $k$-Stirling permutation, we introduce 
one more parameter in $(k+1)$-ary increasing trees.
Let $LR_n$ denote the number of (interior) nodes that have the
property that the path  
to the root consists exclusively of the leftmost or rightmost possible
edge at
each node, i.e., the edge in position 1 or $k+1$, and no  
other ``inner'' edges. Subsequently, we will call such nodes
\emph{left-right nodes}. 
The root is trivially a left-right node.

\begin{theorem}
\label{STIRPthe1}
Let $k\ge1$.
Under the bijection in \refT{STIRPprop1T},
the numbers of $j$-ascents $X_{n,j}$, $j$-descents $Y_{n,j}$ and
$j$-plateaux $Z_{n,j}$ in a $k$-Stirling permutation of order $n$
coincide with the (shifted) numbers of 
$j$-children $D_{n,j}$,  
and $j$-leaves $L_{n,j}$
in a $(k+1)$-ary increasing tree of order $n$ by the formulas
\begin{align*}
  X_{n,j}&=D_{n,j+1}, &&1\le j\le k,
\\
  Y_{n,j}&=D_{n,j}, &&1\le j\le k,
\\
  Z_{n,j}&=L_{n,j+1}=n-D_{n,j+1}, &&1\le j\le k-1.
\end{align*}
As a consequence, for the numbers of ascents, descents and plateaux,
\begin{align*}
  X_{n}&=n-D_{n,1}=L_{n,1},
\\
  Y_{n}&=n-D_{n,k+1}=L_{n,k+1},
\\
  Z_{n}&=\sum_{j=2}^k L_{n,j}.
\end{align*}
Furthermore, the number of blocks $S_n$ in a $k$-Stirling permutations
of order $n$ coincides with the number 
of left-right nodes in the corresponding  
$(k+1)$-ary increasing trees of order $n$,
\begin{equation}\label{slr}
S_n= LR_n.
\end{equation}
\end{theorem}

\begin{proof}
Using the stated bijection we observe that a $(j+1)$-child, $1\le j\le k$, 
corresponds to a $j$-ascent, since the step from the parent node $v$
to the $(j+1)$-child $u$ corresponds  
to having recorded $j$ times the label of the parent $v$ and then
another label $w$, with $w\ge u$. Similar considerations prove the
results for $j$-descents and $j$-plateaux. 
The results for $X_n$, $Y_n$, $Z_n$ then follow from
\eqref{xyzx}--\eqref{xyzz} and \eqref{dla}. 

Concerning the connection between blocks and left-right nodes we make
the following observation. 
Starting with a $(k+1)$-ary increasing tree, and inserting nodes one
after another, we note that only a leftmost or rightmost child leads
to a new  
block in the corresponding $k$-Stirling permutation. Hence, the number
of left-right nodes is equal  
to the number of blocks, since we start with a single block
$1^k$ and a single left-right node (the root).
\end{proof}

\begin{remark}
Note that the number of leaves in $(k+1)$-ary increasing trees of
order $n\ge2$  
corresponds to the number of locally maximal substrings $l^k=l\dotsm
l$, i.e.~substrings $il^kj$, with $0\le i,j<l$, for $2\le l \le n$, in
$k$-Stirling permutations of order $n$, which can also be seen from
the bijection.  
\end{remark}

\begin{remark}
  In the case $k=2$ we thus have the symmetric situation that
  $X_n=L_{n,1}$,
  $Y_n=L_{n,3}$ and $Z_n=L_{n,2}$, which by \refT{STIRPtheExch} below
  gives a new proof that $X_n$, $Y_n$ and $Z_n$ have the same
  distribution, and further are exchangeable, as shown by
  \cite{Bona2007} and \cite{Jan2008}. We see also that this will not
  hold for larger $k$, see for example \refT{Tmean}.
\end{remark}

  For $k=2$, \refT{STIRPprop1T} gives a bijection between \Sp{} and
  ternary increasing trees, while Janson \cite{Jan2008} gives a
  bijection with plane recursive trees of order $n+1$. 
(These are  related by a bijection given in \refS{Sfurther}.)
Next we will show that also for $k>2$, there is a suitable family of
  generalized plane recursive trees that is closely related to
  $k$-Stirling permutations. 

\begin{definition}
\label{STIRPprop2D}
For  $k\ge 2$,
the family of \emph{$k$-plane recursive trees} 
is specified by the degree-weight generating function
$\varphi(t)={(1-(k-1)t)^{-\frac{1}{k-1}}}$, i.e.~it is the
family of generalized plane recursive trees with
 $\varphi_0=1$, $c_1=k$ and $c_2=-(k-1)$.
Explicitly, $\gf_d=\frac1{d!}\prod_{l=1}^d\bigpar{(k-1)(l-1)+1}$.
Consequently, by \eqref{t} and \eqref{tn},
the generating function $T(z)$ and the 
total weight $T_n$ are given by
\begin{equation*}
T(z)=\frac{1}{k-1} \Big(1- (1-kz)^{\frac{k-1}{k}}  \Big),
\qquad T_{n+1}= \prod_{l=1}^{n}(k(l-1)+1),\,\text{with}\enspace
T_1=T_2=1,
\end{equation*}
$\ga=\frac1{k-1}$,
and the probability of attaching to node $v$ at step $i+1$ is given by
$p(v)=\frac{d(v)+\frac{1}{k-1}}{\frac{ki}{k-1}-1}$. 
\end{definition}

For $k=2$, these are the plane recursive trees in \refE{Eplane}.

\begin{remark}
We did not succeed in finding a bijective correspondence between
\kSp{s} and $k$-plane recursive trees, 
in the case of $k>2$, 
generalizing the bijection in \cite{Jan2008} for
$k=2$,
since for $k>2$ it seems difficult to obtain a
combinatorial interpretation 
of the weights of the trees. 
We leave this as an open problem.
However, the distribution of the leaves
still coincides with the distribution of the number of ascents or descents. 
\end{remark}

\begin{theorem}
\label{STIRPthe2}
The number (total weight) of $k$-plane recursive trees of order $n+1$
equals the number of 
$k$-Stirling permutations of order $n$, $T_{n+1}=Q_n$.
Moreover, the distribution of 
the number $\xL_{n+1}$ of
leaves of $k$-plane recursive trees of order $n+1$ coincides with the
distribution of the number $X_n$ of
ascents (descents) of $k$-Stirling permutations of order $n$.
\end{theorem}

\begin{proof}
The first part is already shown. 

The second part is trivial for $n=1$, with one leaf and one ascent.
We proceed by induction, and suppose that the relation is true for
$n$: $\xL_{n+1}\eqd X_n$.
We observe that adding the new node labelled $n+2$ to a leaf
does not change the number of leaves, whereas adding the new node at
any other place gives rise to a new leaf. 
Further, by the formula for $p(v)$ above with d(v)=0, the probability of adding
node $n+2$ to a given leaf in a tree of order $n+1$ is 
$p(v)=\frac{\frac{1}{k-1}}{\frac{k(n+1)}{k-1}-1}=\frac1{kn+1}$. 
Hence, conditioned on the number of leaves $\xL_{n+1}$ being $m$,
we have $\xL_{n+2}=m$ or $m+1$ with
\begin{equation}
  \label{l}
\P(\xL_{n+2}=m\mid \xL_{n+1}=m)
=\frac m{kn+1}.
\end{equation}

Similarly, when adding a string $(n+1)^k$ to a \kSp{} of order $n$, we
will always create a new ascent, and we will destroy one if and only
if we add the string at an ascent. Since there are $kn+1$ gaps where
the new string can be added, conditioned on the number $X_n$ of ascents being
$m$, 
we have $X_{n+1}=m$ or $m+1$ with
$\P(X_{n+2}=m\mid X_{n+1}=m)=\frac m{kn+1}$. 
This is the same relation as \eqref{l}, and thus $\xL_{n+2}\eqd
X_{n+1}$, which verifies the induction step.
\end{proof}

\begin{remark}
The distribution of the number of leaves is fairly well studied.
Let $T(z,v)=\sum_{n\ge 1}T_{n,m}\frac{z^n}{n!}v^m$ denote the bivariate generating function of the number of
$k$-plane recursive trees having exactly $m$ leaves, also encoding the number $k$-Stirling permutations of order $n-1$ having $m$ descents. 
Bergeron et al.~\cite{BerFlaSal1992} determined the generating function $T(z,v)$ by the implicit equation 
\begin{equation*}
\int_{0}^{T}\frac{dt}{(v-1)\varphi_0 + \varphi(t)}=z,
\end{equation*}
Note that the implicit equation is true for a much larger class of increasing trees; moreover 
one may derive the normal limit of the number of leaves from the implicit equation above, see \cite{BerFlaSal1992}.
\end{remark}

\subsection{$k$-bundled increasing trees and $k$-bundled Stirling permutations}

\begin{definition}
\label{Dkbundled}
For  $k\ge 0$,
the family of 
\emph{$(k+1)$-bundled increasing trees} is specified by the
degree-weight generating function 
$\varphi(t)=\frac{1}{(1-t)^{k+1}}$, 
i.e.~it is the family of generalized plane recursive trees with
$\varphi_0=1$, $c_1=k+2$ and $c_2=-1$.
Explicitly, $\gf_j=\binom{k+j}j$.
Consequently, by \eqref{t} and \eqref{tn},
the generating function $T(z)$ and the 
total weight $T_n$ are given by
\begin{equation*}
T(z)=1-(1-(k+2)z)^{\frac1{k+2}} ,\qquad T_{n}= \prod_{l=1}^{n-1}(l(k+2)-1),
\end{equation*}
$\ga=k+1$,
and the the probability attaching to node $v$ at step $i+1$ is given
by $p(v)=\frac{d(v)+k+1}{(k+2)i-1}$.  
\end{definition}

\begin{remark}\label{Rkbundled}
One may think of $(k+1)$-bundled increasing trees of order $n$ as
consisting of a root node labelled $1$ which has $k+1$ positions, 
with a (possibly empty) sequence of labelled $(k+1)$-bundled
increasing trees  attached to each position (with disjoint sets of
labels, forming a partition of $\{2,\dots,n\}$). Equivalently, one
may think of each node as having $k$ separation walls, which can be
regarded as a special type of edges.

Note that the $1$-bundled increasing trees are just ordinary plane
recursive trees, cf.~\refE{Eplane}, and that the bijection stated below
also holds for this case, which corresponds to the result of
\cite{Jan2008} that
$\mathcal{B}_n(1)=\cT_n\cong \mathcal{Q}_{n-1}(2)$,
since obviously $\bcQ_n(0)\cong\cQ_{n-1}(2)$ by relabelling.
\end{remark}

\begin{theorem}
\label{STIRPprop3T}
The family $\mathcal{B}_n=\mathcal{B}_n(k+1)$ of
$(k+1)$-bundled increasing trees of order $n$ is
in a natural bijection with $k$-bundled Stirling permutations,
$\mathcal{B}_n(k+1)\cong \overline{\mathcal{Q}}_n(k)$. 
\end{theorem}

\begin{proof}
We proceed as before using a depth-first walk. We label each auxilliary separation wall of a node labelled $v$ by the label of the node $v$. 
Moreover, we label any (proper) edge by the label of the child. 
Hence, at any time, any node has at least $k$ outgoing edges, thinking
of the walls as a special type of edges.  
Now we perform the depth-first walk and code the $k$-bundled increasing
tree by the sequence of the labels visited on the edges, under the additional rule that a label 
on a separation wall only contributes once. 
Since every proper edge is traversed twice, and every label except 1
occurs on exactly one proper edge,
a $(k+1)$-bundled increasing tree of order $n$ is 
encoded by a string of $(k+2)(n-1)+k$ integers, where each of the labels
$2,\dots, n$ appears exactly $k+2$ times and label $1$ appears $k$
times. In other words, the code is 
a permutation of the multiset $\{1^k, 2^{k+2},\dots, n^{k+2}\}$. Note that
for each $i$, $1\le i \le n$, the elements occurring between the two
occurrences of $i$ are larger than $i$, since we can only visit
nodes with higher labels.
Hence the code is a \kbSp.
Moreover, adding a new node $n+1$ at one of the 
$(k+2)(n-1) + k+1$ possible places
corresponds to
inserting the $(k+2)$-tuple $(n+1)^{k+2}$
in the code, at one of $(k+2)(n-1) + k+1$ possible places. This shows
that the code determines the $(k+1)$-bundled 
increasing tree uniquely and that the coding is a bijection.  
See Figure~\ref{STIRPfig2} for an illustration.
\end{proof}

\begin{figure}[htb]
\centering
\includegraphics[angle=0,scale=0.7]{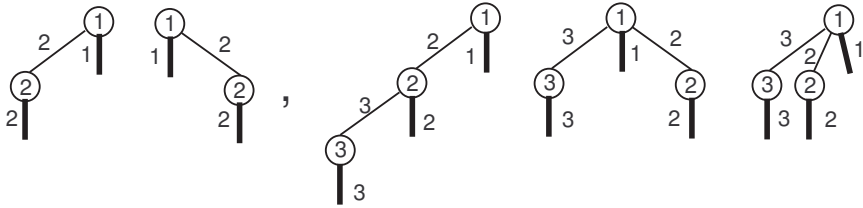}
\caption{The two 2-bundled increasing trees of order $2$ encoded by
  2221, 1222; Three 2-bundled increasing trees of order $3$ encoded by
  the sequences 2333221, 3331222 and 3332221.} 
\label{STIRPfig2}
\end{figure}

Next we relate the distribution of ascents, descents and plateaux in
$k$-bundled Stirling permutations with certain parameters in
$(k+1)$-bundled increasing trees. In order to do so we introduce
three parameters for a  $(k+1)$-bundled increasing tree $\tau$.
The parameter $B_A$ counts the number
of ascents in the bundles of $\tau$, plus the number of non-empty
bundles, plus 1 if the first bundle of the root is empty, where an
ascent in a bundle occurs if the root of a subtree is smaller then
the root of the next subtree, going from left to right. The
parameter $B_D$ counts the number descents in the bundles of $\tau$,
plus the number of non-empty bundles, plus 1 if the last bundle of
the root is empty, where a descent in a bundle occurs if the root of
subtree is larger then its neighbour. The number $B_E$ counts the
number of empty bundles of the nodes with labels larger than one
plus the number of empty inner bundles of the root. With these
definitions, the following correspondences are straightforward.

\begin{theorem}
\label{STIRPthe3}
Under the bijection in \refT{STIRPprop3T},
the numbers of ascents, descents and plateaux in a $k$-bundled Stirling
permutation of order $n$
coincide with the parameters  $B_A$, $B_D$
and $B_E$ in a $(k+1)$-bundled increasing tree of order $n$.
\end{theorem}

\begin{remark}
Note that the number of leaves in $(k+1)$-bundled increasing trees
of order $n$ corresponds to the number of sequences of the form
$l^{k+2}=l\dotsm l$, with $2\le l \le n$, in $k$-bundled Stirling
permutations of order $n$, as for $k$-ary increasing trees.
Moreover, the parameter ``number of descendants of node $j$'' in  a
$(k+1)$-bundled increasing tree of order $n$, with $2\le j\le n$,
counts the number of different entries $l$ with $j<l\le n$  between
the first and the last occurrence of $j$ in the corresponding
$k$-bundled Stirling permutation of order $n$.
\end{remark}

\section{Further bijections}
\label{Sfurther}
The bijections of Theorem~\ref{STIRPprop1T} (with $k=2$) and \cite{Jan2008}
(or \refT{STIRPprop3T} with $k=0$)
imply a bijection between ordinary plane recursive trees of order $n+1$ and 
ternary increasing trees of order $n$, using the connections to
$2$-Stirling permutations. 
In the following we will give two direct bijections, which 
both encompass this bijection between plane recursive trees and
ternary increasing trees. 

First we give a bijection 
between sequences of $k$-bundled increasing trees and $(k+2)$-ary
increasing trees,  
which for $k=1$ just gives the desired bijection. 

Let $\Seq(\mathcal{B})_n=\Seq(\mathcal{B})_n(k)$ denote the family of
sequences of $k$-bundled increasing trees with total order $n$,
labelled with disjoint sets of labels forming a partition of
\set{1,\dots,n}.
(Note that our notation slightly abuses the common sequence notation
$\Seq$ of combinatorial objects, since we also assume properly  
distributed labels.) 

\begin{remark}
  \label{RSEQ}
By introducing a new
root labelled 0, connecting all roots 
of the sequence with the new root, and performing a proper
relabelling,
$\Seq(\cB)_n$ is in bijection with the family of increasing plane
trees of order $n+1$ where each node except the root is $k$-bundled as
in \refD{Dkbundled}. 
(Equivalently, 
$\Seq(\cB)_n$ is in bijection with the family of $k$-bundled
increasing trees of order $n+1$ where the root has only the first
bundle non-empty.)
\end{remark}

\begin{theorem}
\label{STIRPprop4T}
The family $\Seq(\mathcal{B})_n=\Seq(\mathcal{B})_n(k)$ of sequences
of $k$-bundled  
increasing trees of total order $n$ is in bijection with
$\mathcal{A}_n(k+2)$, the family of $(k+2)$-ary increasing trees of
order $n$:
$\Seq(\mathcal{B})_n(k) \cong \mathcal{A}_n(k+2)$.
\end{theorem}

\begin{remark}\label{RSEQ2}
Recall that $1$-bundled increasing trees  are exactly plane recursive trees.
Moreover, in the case of $k=1$, 
the bijection in \refR{RSEQ} is the standard bijection between
sequences of plane recursive trees of total order $n$ and   
plane recursive trees of order $n+1$; hence $\Seq(\cB)_n(1)\cong\cT_{n+1}$.
See Figure~\ref{STIRPfig3} for an illustration.
It is easily seen that, for $k=1$, the bijection
$\cT_{n+1}\cong\Seq(\cB)_n(1)\cong\cA_n(3)$ constructed in the proof
below yields the correspondence between the two bijections
 $\cA_n(3)\cong \cQ_n(2)$ in \refT{STIRPprop1T} and
$\cT_{n+1}\cong\cQ_n(2)$ in \cite{Jan2008} or
$\cB_{n+1}(1)\cong\bcQ_{n+1}(0)\cong\cQ_{n}(2)$ in \refT{STIRPprop3T}.
\end{remark}
 
\begin{proof}
We use a recursive construction, see Figure~\ref{STIRPfig3}.
For a given sequence of $k$-bundled 
increasing trees, we choose in the first step the tree of the sequence with node labelled 1: this node is going be the root of the $(k+2)$-ary increasing tree. 
Since a $(k+2)$-ary increasing trees has $k+2$ (possibly empty)
subtrees $S_1,\dots,S_{k+2}$, going from left to right, we proceed as
follows. 
The sequence of $k$-bundled 
increasing trees to the left of the tree with root $1$ forms
(recursively) the subtree
$S_1$, conversely the sequence of  
$k$-bundled increasing trees to the right of the tree with root $1$
forms the subtree $S_{k+2}$. 
The $k$ bundles, possibly empty, attached to the tree with root
labelled $1$, form the subtrees $S_2,\dots,S_{k+1}$ of the $(k+2)$-ary
increasing tree. 
Now we can proceed recursively, since the bundles are themselves just
sequences of $k$-bundled increasing trees.

\smallskip
Conversely,
starting with a $(k+2)$-ary increasing tree of order $n$, we recursively build a sequence of $k$-bundled increasing trees as follows. 
In the first step we build a tree with root node labelled 1. The sequence to the left of the tree with root labelled 1
is built from the subtree $S_1$ of the $(k+2)$-ary increasing tree of
order $n$, the sequence on the right from the subtree $S_{k+2}$,  
and the $k$ bundles are built from the subtrees $S_2,\dots,S_{k+1}$.
We proceed recursively until the sequence is constructed.
Note that during this process, we connect any leftmost or rightmost
child of a node $v$ to the same  parent as $v$.
\end{proof}

\begin{figure}[htb]
\centering
\includegraphics[angle=0,scale=0.65]{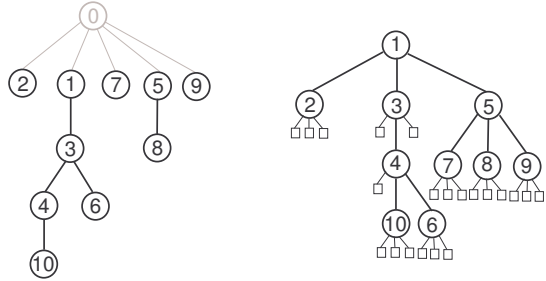}
\caption{A sequence of 1-bundled increasing trees of order 10, or
  equivalently a plane recursive tree of order 11,  
and the corresponding ternary increasing tree of order 10.}
\label{STIRPfig3}
\end{figure}

Next we consider a bijection between $k$-bundled increasing trees
and so-called $F_{k,k+2}$-increasing trees. The family of
$F_{k,k+2}$-increasing trees consists of modified $(k+2)$-ary
increasing trees: any node except the root of a
$F_{k,k+2}$-increasing tree has $k+2$ labelled positions where
children may be attached, whereas the root has only $k$ positions
(and thus outdegree bounded by $k$). Note that for $k=1$, the root
has a single child and that chopping off the root yields a simple
bijection between $F_{1,3}$-increasing trees of order $n+1$ and
ternary increasing trees of order $n$. Thus the statement below
implies for $k=1$ a bijection between ternary increasing trees and
plane recursive trees, $\cA_n(3)\cong\cB_{n+1}(1)=\cT_{n+1}$, which
is just the bijection discussed in \refR{RSEQ2}.

\begin{theorem}
\label{STIRPprop5T}
The family $\mathcal{F}_n=\mathcal{F}_n(k)$ of $F_{k,k+2}$-increasing trees of order $n$, is in bijection 
with the family of $k$-bundled increasing trees of order $n$, $\mathcal{F}_n(k)\cong \mathcal{B}_n(k)$, $k\ge 1$.
\end{theorem}

\begin{proof}
For a given $k$-bundled increasing tree of order $n$, we simply apply
$k$ times the bijection
between sequences of $k$-bundled increasing trees and $(k+2)$-ary increasing trees to the $k$ bundles attached to the root and 
the $k$ positions of the root of the $F_{k,k+2}$-increasing tree.
\end{proof}

\begin{remark}
To give an overview, we have provided the following bijections in
Theorems~\ref{STIRPprop1T},~\ref{STIRPprop3T},~\ref{STIRPprop4T} and
~\ref{STIRPprop5T}. 
\begin{equation*}
 \mathcal{A}_n(k+1) \cong 
 \begin{cases} 
 \mathcal{Q}_n(k),\\
 \Seq(\mathcal{B})_n(k-1),
 \end{cases}
\mathcal{B}_n(k+1) \cong 
 \begin{cases} 
 \overline{\mathcal{Q}}_n(k),\\
 \mathcal{F}_n(k+1).
 \end{cases}
\end{equation*}
It is also possibly to give bijections $\mathcal{Q}_n(k) \cong
\Seq(\mathcal{B})_n(k-1)$ and  
$\overline{\mathcal{Q}}_n(k)\cong \mathcal{F}_n(k+1)$, by simple modifications of the stated bijections.
\end{remark}

\begin{remark}
The families $\Seq(\mathcal{B})(k)$ of sequences $k$-bundled increasing
trees and $\mathcal{F}(k)$ of $F_{k,k+2}$-increasing trees
are non-standard in the sense that they are not part of the
characterization given by Panholzer and Prodinger~\cite{PanPro2005+}. 
However, the counting problem concerning such tree families can be
treated in a general manner, which will be discussed elsewhere. 
\end{remark}

\section{The distribution of $j$-ascents, $j$-descents and
 $j$-plateaux}
We are interested in the joint asymptotic distribution of $j$-ascents
$X_{n,j}$, $j$-descents $Y_{n,j}$ and $j$-plateaux $Z_{n,j}$ in a
$k$-Stirling permutations of order $n$, or equivalently  
in the joint distribution of $j$-children $D_{n,j}$ and $j$-leaves
$L_{n,j+1}$ 
in $(k+1)$-ary increasing trees of order $n$. Following Janson \cite{Jan2008} we use a
(generalized) P\'olya urn model, see \cite{Jan2004}.

\subsection{An urn model for the exterior leaves}
Since we already know from \eqref{ld} that $n-D_{n,j}=L_{n,j}$, 
we can restrict ourselves to the study of the exterior nodes. 
We will use the following urn model.
\smallskip

\begin{urn}\label{UA}
Consider an urn with balls of $k+1$ colours, and let 
$(L_{n,1},\dots,L_{n,k+1})$ be
the number of balls of each colour at time $n$. At each time step, draw one
ball at random from the urn, discard it, and add one new ball of each colour.
Start with $(L_{1,1},\dots,L_{1,k+1})=(1,1,\dots,1)$.
Note that the vector $(L_{n,1},\dots,L_{n,k+1})$ \emph{exactly} 
coincides (in distribution) with the numbers of the exterior nodes of
types $1,\dots,k+1$ in a random $(k+1)$-ary increasing tree, see \refS{SSinck}.
\end{urn}

\refUA{} is completely symmetric in the $k+1$ colours, and we thus
immediately see the following.

\begin{theorem}\label{STIRPtheExch}
For each $n\ge 1$, the distribution of $(L_{n,1},\dots,L_{n,k+1})$ is
exchangeable, 
i.e., invariant under any permutation of the $k+1$ variables.
\end{theorem}

It is customary and convenient to formulate generalized P\'olya urns using
drawings with replacement. In the case of \refUA, we thus restate the
description 
above and say instead that we draw a ball and replace it together with
one ball each of the $k$ other colours. In other words, \refUA{} is described
by the $(k+1)\times(k+1)$ replacement matrix
\begin{equation*}
    A = (1-\delta_{i,j})_{1\le i,j\le k+1}=\left(
    \begin{matrix}
        0 & 1 & 1 & \cdots & 1 & 1 & 1 \\[-1ex]
        1 & 0 & 1 & \ddots & \ddots & \ddots & 1 \\[-1ex]
        1 & 1 & 0 & \ddots & \ddots & \ddots & 1 \\[-1ex]
        \vdots & \ddots & \ddots & \ddots & \ddots &
        \ddots & \vdots \\[-1ex]
        1 & \ddots & \ddots & \ddots & 0 & 1 & 1 \\[-1ex]
        1 & \ddots & \ddots & \ddots & 1 & 0 & 1 \\
        1 & 1 & 1 & \cdots & 1 & 1 & 0
    \end{matrix}
    \right),
\end{equation*}
where $\delta_{i,j}$ denotes the Kronecker delta.

\subsection{Means}
By \refT{STIRPtheExch}, the variables $L_{n,j}$, $j=1,\dots,k+1$, have
the same mean, and since their sum is $kn+1$ by \eqref{dla}, we see
that they each have mean $\frac{kn+1}{k+1}$.
By \eqref{ld} and \refT{STIRPthe1}, we obtain the following exact
formulas for the means.
\begin{theorem}
  \label{Tmean}
The following hold, for $n\ge1$ and $k\ge1$:
  \begin{align*}
\E L_{n,j}&=\frac{kn+1}{k+1}, && 1\le j\le k+1,
\\
\E D_{n,j}&=\frac{n-1}{k+1}, && 1\le j\le k+1,
\\
\E X_{n,j}&=
\E Y_{n,j}=
\frac{n-1}{k+1}, && 1\le j\le k,
\\
\E Z_{n,j}&=\frac{kn+1}{k+1}, && 1\le j\le k-1,
\\
\E X_{n}&=
\E Y_{n}=
\frac{kn+1}{k+1},
\\
\E Z_{n}&=(k-1)\frac{kn+1}{k+1}.  && 
  \end{align*}
\end{theorem}

\subsection{Asymptotic distribution of $j$-ascents, $j$-descents and
  $j$-plateaux} 

We use the urn model \refUA{} to obtain asymptotic normality. We begin
with a general result.

\begin{theorem}
  \label{Turn}
Consider an urn with balls of $q\ge2$ colours, where at each step one
ball is drawn
at random and discarded, and one ball of each colour is added.
If $N_{n,j}$ is the number of balls of colour $j$ after $n$ steps,
then, for any initial values $N_{0,1},\dots,N_{0,q}$,
  \begin{equation*}
	\frac{N_{n,j}-\frac{q-1}q n}{\sqrt n} \dto \zeta_j,
  \end{equation*}
jointly for $j=1,\dots,q$, where $\zeta_j$ are jointly normal random
variables with means $0$ and (co)variances
\begin{align*}
\Cov(\zeta_i,\zeta_j)=\frac{q-1}{q^2(q+1)}(q\gd_{i,j}-1), \qquad 1\le i,j\le q.
\end{align*}
\end{theorem}

Note that $\sum_1^q\zeta_j=0$, for example because $\sum_j N_{n,j}$ is
deterministic. 

\begin{proof}
  This urn has replacement matrix 
$A=(1-\gd_{i,j})_{i,j=1}^q=(1)_{i,j=1}^q-I$.
Since the rank 1 matrix $(1)_{i,j=1}^{q}$ has one eigenvalue $q$ and $q-1$
eigenvalues 0, $A$ has largest eigenvalue $\gl_1=q-1$ and $q-1$ eigenvalues
$-1$. 
Theorem 3.22 in \cite{Jan2004} applies, with
$v_1=(\frac1q,\dots,\frac1q)$, and shows joint convergence (in distribution)
to normal
variables $\zeta_j$ with mean 0 and a certain covariance matrix $\Sigma$.
The formula for $\Sigma$ in \cite[Theorem 3.22]{Jan2004} is
complicated, so we use 
\cite[Lemma 5.4]{Jan2004}, 
with $a=(a_i)_1^q=(1)_1^q$ and
$m=\gl_1=q-1$, which yields $\Sigma=m\Sigma_I$, with $\Sigma_I$
defined in \cite[(2.15)]{Jan2004}.
Here $\xi_i=(\xi_{i,j})_{j=1}^q=(1-\gd_{i,j})_{j=1}^q$,
$B_i=(\xi_{i,j}\xi_{i,l})_{j,l=1}^q$, $v_{1i}=1/q$ and, using the symmetry,
$B=(b_{ij})_{i,j=1}^q$ with $b_{i,i}=\frac{q-1}q$ and
$b_{i,j}=\frac{q-2}q$, $i\neq j$; hence $B=\frac{q-2}q A+\frac{q-1}q I$.
Further, $P_I$ is the projection onto the eigenspace of $A$ for the
eigenvalue $-1$, and thus $P_I=\frac{(q-1)I-A}q$. 
Further, on this eigenspace $A=-I$ and thus $B=\frac1q I$, and 
\cite[(2.15)]{Jan2004} yields, noting that all involved matrices are
symmetric and commute,
\begin{equation*}
  \begin{split}
  \Sigma_I
&=\int_0^\infty P_I e^{sA}Be^{sA'}P_I'e^{-\gl_1s}\,ds	
=\frac1q P_I\int_0^\infty e^{-s-s-(q-1)s}\,ds
=\frac1{q(q+1)}P_I
\\&
=\frac{(q-1)I-A}{q^2(q+1)}
=\lrpar{\frac{q\gd_{i,j}-1}{q^2(q+1)}}_{i,j=1}^q.	
  \end{split}
\end{equation*}
Recalling that $\Sigma=m\Sigma_I=(q-1)\Sigma_I$, we obtain the result.
\end{proof}

\begin{remark}\label{Rurn}
  Similar calculations show, more generally, that if we at each step
  add a fixed number $s_i$ balls of colour $i$, $i=1,\dots,q$,
  independently of the colour of the drawn and discarded ball, then
$n^{-1/2}\bigpar{N_{n,i}-\frac{\sum_l s_l-1}{\sum_l s_l}s_i n}\dto\zeta_i$,
  jointly, where $\zeta_i$  are jointly normal variables with means
  $0$
and
\begin{equation*}
  \Cov(\zeta_i,\zeta_j)
=
\frac{\sum_l s_l-1}{\sum_l s_l+1}
\lrpar{\frac{s_i}{\sum_l s_l}\gd_{i,j}
-\frac{s_is_j}{(\sum_l s_l)^2}
}.
\end{equation*}
As an example, the numbers $X_n$, $Y_n$ and $Z_n$ of ascents, descents
and plateaux in a random \kSp{} can
be seen as such an urn with replacement vector $(1,1,k-1)$,
which yields an alternative proof of the limit distribution found
above for them. 
\end{remark}

We apply \refT{Turn}, with $q=k+1$, to \refUA{} and obtain using \eqref{ld}
and \refT{STIRPthe1} the following.
\begin{theorem}
  \label{Tnormal}
Let $k\ge1$ and
let $\zeta_j$, $j=1,\dots,k+1$, be jointly normal random
variables with means $0$ and (co)variances
\begin{align*}
\Cov(\zeta_i,\zeta_j)=\frac{k}{(k+1)^2(k+2)}((k+1)\gd_{i,j}-1), \qquad
1\le i,j\le k+1,
\end{align*}
in particular $\Var(\zeta_j)=\frac{k^2}{(k+1)^2(k+2)}$. Note that this
implies $\sum_{j=1}^{k+1}\zeta_j=0$.
Then, the following holds, jointly for all variables,
\begin{align*}
\frac{L_{n,j}-\frac{k}{k+1} n}{\sqrt n} &\dto \zeta_j,  
&& 1\le j\le k+1,
\\
\frac{D_{n,j}-\frac{1}{k+1} n}{\sqrt n} &\dto -\zeta_j,  
&& 1\le j\le k+1,
\\
\frac{X_{n,j}-\frac{1}{k+1} n}{\sqrt n} &\dto \xi_j:=-\zeta_{j+1},  
&& 1\le j\le k,
\\
\frac{Y_{n,j}-\frac{1}{k+1} n}{\sqrt n} &\dto \eta_j:=-\zeta_{j},  
&& 1\le j\le k,
\\
\frac{Z_{n,j}-\frac{k}{k+1} n}{\sqrt n} &\dto \zeta_{j+1},  
&& 1\le j\le k-1,
\\
\frac{X_{n}-\frac{k}{k+1} n}{\sqrt n} &\dto \xi:=\zeta_{1},  
\\
\frac{Y_{n}-\frac{k}{k+1} n}{\sqrt n} &\dto \eta:=\zeta_{k+1},  
\\
\frac{Z_{n}-\frac{k(k-1)}{k+1} n}{\sqrt n} &\dto 
\zeta:=\sum_{j=2}^{k}\zeta_{j}=-\xi-\eta.  
\end{align*}
\end{theorem}
A simple calculation shows that the covariance matrix of
$(\xi,\eta,\zeta)$ is (cf.\  \refR{Rurn}) 
\begin{align*}
\left(\begin{matrix}
\frac{k^2}{(k+1)^2(k+2)}&-\frac{k}{(k+1)^2(k+2)}&-\frac{k(k-1)}{(k+1)^2(k+2)}
\\
-\frac{k}{(k+1)^2(k+2)} &\frac{k^2}{(k+1)^2(k+2)} &-\frac{k(k-1)}{(k+1)^2(k+2)}
\\
-\frac{k(k-1)}{(k+1)^2(k+2)}&-\frac{k(k-1)}{(k+1)^2(k+2)}
   &\frac{2k(k-1)}{(k+1)^2(k+2)}
\end{matrix}\right).
\end{align*}

For $k=2$, this yields the 
univariate limit theorems by Bona \cite{Bona2007} 
and the multivariate limit theorem by Janson \cite{Jan2008} 
for $X_n,Y_n,Z_n$.

For $k=1$, the result for $X_n$ or $Y_n$ reduces to the
classical result on the asymptotics of the number of ascents or
descents in a random permutation. (In this case $Z_n=0$.)

\section{The distribution of the number of blocks}

The number $S_n$ of blocks 
in a random $k$-Stirling permutation is described by another urn model.

\begin{urn}\label{UB}
This urn has balls of two colours, black and white. 
At each time step, draw a ball at random from the urn, replace it and
add $k$ further balls: 
if the drawn ball was black, add $k$ black balls;
if the drawn ball was white, add 1 white ball and $k-1$ black.
Let $B_n$ and $W_n$ be the numbers of black and white balls in the urn
at time $n$, and start with $W_n=2$, $B_n=k-1$.
\end{urn}

We thus have $B_n+W_n=kn+1$ balls in the urn at time $n$,  and it is
easily seen that the number of white balls can be interpreted as the
number of gaps between the blocks, or first or last, in a random
\kSp{} of order $n$,
i.e.~as the number of gaps where addition of a string $(n+1)^k$
create a new block. This is one more than the number of blocks, and
thus we have the equality in distribution 
\begin{equation}\label{sw}
  S_n\eqd W_n-1.
\end{equation}

\refUB{} is thus a $2\times 2$ \gpu{}  
with ball replacement matrix $M=\bigl(\begin{smallmatrix} k & 0 \\ k-1
  & 1\end{smallmatrix}\bigr)$.
This urn model is a special case of the triangular $2\times 2$ urn
models analysed in detail by Janson~\cite{Jan2005b}, where the
asymptotic distribution is given. 
The special case of balanced triangular $2\times 2$ urn models was also
studied by Flajolet et al.~\cite{FlaDuPu2006}. 
(An urn is called balanced if the total number of added balls is
constant, independently  of the observed color.) 
For the special case treated here we can add the exact distribution
using the tree representation, 
the moments of $S_n$, and almost sure convergence.

\begin{theorem}
The probability mass function of the random variable $S_n$ counting
the number of blocks  
in a random $k$-Stirling permutation of order $n$ is given by 
\begin{equation*}
\P\{S_n=m\}=\sum_{\ell=0}^{m}\binom{m}{l}(-1)^{\ell}\frac{\binom{n-\frac{\ell}{k}-1}{n}}{\binom{n+\frac{1}{k}-1}{n}}, 
\qquad 1\le m\le n. 
\end{equation*}
The binomial
moments $\E\binom{S_n+r}{r}$ are given by the explicit formula  
\begin{equation*}
\E\binom{S_n+r}{r}
=
\frac{\binom{n-1+\frac{r+1}k}{n}}{\binom{n-1+\frac{1}k}{n}}
=
(r+1)\frac{\binom{n-1+\frac{r+1}k}{n-1}}{\binom{n-1+\frac{1}k}{n-1}},
\quad r= 1,2,\dots
\end{equation*}
The random variable
$\mathcal{S}_n:=
\frac{\binom{n-1+\frac{1}k}{n-1}}{\binom{n-1+\frac{2}k}{n-1}}(S_n+1)$ 
is a positive martingale and converges almost surely to a limit
$\zzeta$, 
i.e.~$\mathcal{S}_n\alaw \zzeta$,
further
\begin{equation*}
n^{-1/k}{S}_n\alaw \zeta
=\frac{\Gamma(1+\frac1k)}{\Gamma(1+\frac{2}k)}\zzeta.
\end{equation*}
The limits $\zzeta$ and $\zeta$ can be specified by the moments
\begin{equation*}
\E(\zeta^r)=(r+1)!\frac{\Gamma(1+\frac1k)}{\Gamma(1+\frac{r+1}k)},\quad
r\ge0. 
\end{equation*}
Further, $\zeta$ has a density function $f(x)$ that can be written as
$f(x)=\Gamma(\frac1k)x^{-k}g(x^{-k})$, $x>0$, 
where $g$ is the density function of a
positive  $\frac1k$-stable distribution with Laplace transform
$e^{-\lambda^{1/k}}$; it is thus given by the series expansion
\begin{equation*}
  f(x)=\frac{\Gamma(\frac1k)}{\pi}\sum_{j=1}^\infty
 (-1)^{j-1} \frac{\Gamma(\frac jk+1)\sin\frac{j\pi}k}{j!}\,x^{j},
\qquad x>0.
\end{equation*}
\end{theorem}

\begin{remark}
The simple structure of the binomial moments and the almost sure
convergence is actually true 
for all balanced triangular urns of the form $M=\bigl(\begin{smallmatrix} \alpha & 0 \\ \beta-\alpha & \beta\end{smallmatrix}\bigr)$,
$0<\alpha<\beta$, which is easily seen to be true by extending the martingale arguments to thie general case.
\end{remark}

\begin{proof}
We use three different approaches to study the block structure $S_n$
in $k$-Stirling permutations or equivalently  
the number of left-right edges $LR_n$ in $(k+1)$-ary increasing trees,
see \eqref{slr}.
In order to obtain the explicit results for the probability
distribution of $S_n$, we analyze $LR_n$.  
We can use the tree decomposition
\eqref{eqnz0} in order to obtain the differential equation
\begin{equation*}
\frac{\partial}{\partial z}T(z,v)=v(1+T(z,v))^2(1+T(z))^{k-1},\quad T(0,v)=0,
\end{equation*}
where $T(z,v)=\sum_{n\ge 1}\sum_{m\ge 1}\P\{S_n=m\}T_n\frac{z^n}{n!}v^m$ denotes 
the bivariate generating function of the numbers $\P\{S_n=m\}T_n$, and
$T(z)=T(z,1)$ is the generating function of total 
weights of $(k+1)$-ary increasing trees.
By \refE{Ed}, $1+T(z)=(1-kz)^{-1/k}$.
Solving the differential equation and adapting to the initial
condition gives the solution 
\begin{equation*}
T(z,v)= \frac{1}{1-v\big(1-(1-kz)^{1/k}\big)}-1.
\end{equation*}
Extraction of coefficients gives then the stated result for the probability mass function.
Moreover, the stated binomial moments may be obtained from the
generating function by extracting coefficients, 
\begin{equation*}
\E\binom{S_n+r}{r}
=\frac{n!}{T_n}[z^nw^r]\frac{1}{1-w}T\big(z,\frac{1}{1-w}\big). 
\end{equation*}

\smallskip

For the almost sure convergence we proceed as follows. 
Let $W_n=S_n+1$ be the number of gaps between blocks, or,
equivalently, the number of white balls in \refUB, see \eqref{sw}.
Let $\mathcal{F}_{n}$ denote the $\sigma$-field generated by the first
$n$ steps.  
Moreover denote by $\Delta_n=W_n-W_{n-1}=S_n-S_{n-1}\in\{0,1\}$ the
increment at step $n$.  
We have 
\begin{equation*}
\E(W_n \mid \mathcal{F}_{n-1})= \E(W_{n-1}+\Delta_n \mid \mathcal{F}_{n-1})
=W_{n-1} +  \E(\Delta_n \mid \mathcal{F}_{n-1}).
\end{equation*}
Since the probability that a new white ball is generated at step $n$
is proportional to the number of existing white balls (at step $n-1$),
we obtain further 
\begin{equation*}
\E(W_n \mid \mathcal{F}_{n-1})
=W_{n-1} +\frac{W_{n-1}}{k(n-1)+1}
=\frac{k(n-1)+2}{k(n-1)+1}W_{n-1},\quad n\ge 2.
\end{equation*}
Hence, 
\begin{equation*}
\E( \mathcal{S}_n \mid \mathcal{F}_{n-1})
=\frac{\binom{n-1+\frac{1}k}{n-1}}{\binom{n-1+\frac{2}k}{n-1}}
\frac{k(n-1)+2}{k(n-1)+1}
W_{n-1}
=\mathcal{S}_{n-1},\quad n\ge 2.
\end{equation*}
Hence, $\mathcal{S}_n$ is a martingale.
Since it is a positive martingale, it converges almost surely to a
limit $\zzeta$. 
By the well-known asymptotic formula $\binom{n+a}n\sim
n^a/\Gamma(a+1)$, for any 
fixed real $a$, 
$\cS_n\sim\frac{\Gamma(1+\frac2k)}{\Gamma(1+\frac{1}k)} n^{-1/k} S_n$ 
(provided $S_n\to\infty$),
and thus $\cS_n\alaw \zzeta$ can also be written
$n^{-1/k}S_n\alaw\zeta$.

More generally, we similarly have for any positive integer $r$
\begin{equation*}
\begin{split}
\E\bigg(\binom{S_n+r}{r}\biggm| \mathcal{F}_{n-1}\bigg)&= \binom{S_{n-1}+r}{r} + \binom{S_{n-1}+r}{r-1}\frac{S_{n-1}+1}{k(n-1)+1}\\
&=\binom{S_{n-1}+r}{r}\frac{n-1+\frac{r+1}{k}}{n-1+\frac1k}.
\end{split}
\end{equation*}
Hence, 
$\binom{S_n+r}{r}
 \frac{\binom{n-1+\frac1k}{n-1}} {\binom{n-1+\frac{r+1}k}{n-1}}$ 
is a martingale, which also leads to the stated result for the moments
 in an alternative way, using  
the recurrence relation for the unconditional expectation.

Letting \ntoo{} in the moment formula yields
\begin{equation*}
\E\binom{S_n+r}{r}
= (r+1) \frac{\binom{n-1+\frac{r+1}k}{n-1}}
	 {\binom{n-1+\frac{1}k}{n-1}}
\sim (r+1) \frac{\Gamma\xpar{1+\frac{1}k}} {\Gamma\xpar{1+\frac{r+1}k}}n^{r/k}
\end{equation*}
which leads to
$\E S_n^r
\sim (r+1)! \frac{\Gamma\xpar{1+\frac{1}k}} {\Gamma\xpar{1+\frac{r+1}k}}n^{r/k}
$. Hence all moments of $n^{-1/k}S_n$ converge, and the limits must be
the moments of $\zeta$. Letting $r\to\infty$, we see that the moments
do not grow too fast so that the moment generating function $\E
e^{t\zeta}$ is finite for all $t$, and thus the distribution is
determined by the moments.

Finally, we use the general results for urn models in Janson \cite{Jan2005b}.
First, \cite[Theorem 1.3(v)]{Jan2005b} 
yields the convergence
$W_n\dto\zeta$ in distribution, 
and \cite[Theorem 1.7]{Jan2005b} yields the moments of $\zeta$ that we
just derived in a different way; note however that 
\cite[Theorem 1.7]{Jan2005b} yields the formula above also for
non-integer $r\ge0$, with the standard interpretation
$(r+1)!=\Gamma(r+2)$.
Furthermore,
\cite[Theorem 1.8]{Jan2005b} 
shows that $\zeta$ has a density and gives the explicit formulas
stated above.
\end{proof}

\section{The sizes of the blocks}

Recall that every block in a \kSp{} begins and ends with the same
label, which we can regard as a label of the block. We order the
blocks in the block decomposition as $\tcK_1,\dots,\tcK_s$ according
to this label (where $s$ is the number of blocks); 
thus $\tcK_1$ is the block extending from the first 1
to the last, $\tcK_2$ is the block formed by the smallest label not in
$\tcK_1$, and so on. We also let $\tK_i:=|\tcK_i|$ denote the size of
the $i$th block in this order, and put $\tcK_i=\emptyset$, $\tK_i=0$
for $i>s$.

Alternatively, we may order the blocks according to decreasing
size. We let $K_1\ge K_2\ge\dots$ be the sizes of the blocks in this
order, again with $K_i=0$ for $i>s$. Thus, $(K_i)_1^\infty$ is the
decreasing rearrangement of $(\tK_i)_1^\infty$.

For a random \kSp{} of order $n$, we use the notations $\tcK_{n,i}$,
$\tK_{n,i}$ and $K_{n,i}$. Note that $\sum_i\tK_{n,i}=\sum_i K_{n,i}=kn$.

To study these sizes we introduce another urn model. Consider first an
urn with balls of two colours, $\tK_{n,1}-1$ white balls representing
the gaps inside the block $\tcK_{n,1}$ and $nk+2-\tK_{n,1}$ black
balls representing the gaps outside. Adding the string $(n+1)^{k}$ at
one of the gaps inside $\tcKn1$ means increasing $\tKn1$ by $k$, and
adding it outside means keeping $\tKn1$ unchanged; hence this is a
P\'olya urn of the original type considered by
Eggenberger and P\'olya \cite{EggPol1923}, \cite{Polya1931},
where we draw a ball at random and replace it together with $k$ balls
of the same colour.
We start with $\tK_{1,1}=k$, and thus $k-1$ white and 2 black balls.

Next, let us study the second block, $\tcKn2$. At the first $n$ where
this is non-empty, we have $k+2$ gaps outside the first block
$\tcKn1$, $k-1$ of them in $\tcKn2$ and 3 outside both blocks.
Let us now ignore the first block and consider an urn with $\tKn2-1$
white balls representing the gaps in $\tcKn2$ and black balls
representing the gaps outside both $\tcKn1$ and $\tcKn2$.
The balls in this urn are drawn at random times (when we do not add to
a gap in $\tcKn1$), but when they are drawn, the urn behaves exactly
as for $\tcKn1$: we replace the drawn ball together with $k$ of the
same colour.

The same argument applies to $\tcKn m$ for any $m\ge2$; if we ignore
the preceding blocks and additions to them, we have the same P\'olya
urn again, but now started with $m+1$ black balls, representing the
gaps outside the first $m$ blocks. We hence make the following
definition.

\begin{urn}
  \label{UC}
This is the standard P\'olya urn with balls of two colours and where
each drawn ball is replaced together with $k$ balls of the same
colour. Let \refUCm{} be the version where we start with $k-1$ white
and $m+1$ black balls, and let $\W_{N,m}$ and $\B_{N,m}$ denote the
  numbers of white and black balls after $N-1$ draws, when the urn
  contains $\W_{N,m}+\B_{N,m}=kN+m$ balls.
\end{urn}

We can thus identify (with the urns \refUCx1, \refUCx2, \dots
independent),
recalling that the balls in urn \refUCx{m+1} correspond to the black
balls in urn \refUCm,
\begin{align*}
  \tKn1&=\W_{n,1}+1,
\\
  \tKn2&=\W_{N_2,2}+1, \quad \text{with}\quad kN_2+2=\B_{n,1},
\\
  \tKn3&=\W_{N_3,3}+1, \quad \text{with}\quad kN_3+3=\B_{N_2,1},
\end{align*}
and so on.

\begin{theorem}\label{TGEM}
  There exists a sequence of independent beta distributed random
  variables
$\gb_m\sim\Beta(\frac{k-1}k,\frac{m+1}k)$ such that
  \begin{equation}\label{tgem}
	\frac1{kn}(\tKn1,\tKn2,\dots)
\alaw
(\gb_1,(1-\gb_1)\gb_2, (1-\gb_1)(1-\gb_2)\gb_3,\dots).
  \end{equation}
\end{theorem}

\begin{proof}
  The basic limit theorem for P\'olya urns says that, as $N\to\infty$,
  \begin{equation*}
\frac{\W_{N,m}}{kN}\alaw\gb_m\sim\Beta\Bigpar{\frac{k-1}k,\frac{m+1}k}
  \end{equation*}
and thus 
$
\frac{\B_{N,m}}{kN}\alaw1-\gb_m$.
(This is already in P\'olya \cite{Polya1931} for convergence in
distribution. See, for example, \cite{JohnKotz1977} or 
\cite[Section 11]{Jan2005b}.) 
Consequently,
\begin{align*}
  \frac{\tKn1}{kn}&=\frac{\W_{n,1}+1}{kn}\alaw\gb_1,
\\
  \frac{\tKn2}{kn}&=\frac{\B_{n,1}}{kn}\cdot\frac{\W_{N_2,2}+1}{kN_2+2}
\alaw(1-\gb_1)\gb_2,
\end{align*}
and so on.
\end{proof}

Note that both sides of \eqref{tgem} are elements of $\cP$,
the space of non-negative sequences $(p_i)_1^\infty$ with
$\sum_ip_i=1$; $\cP$ can be seen as the space of probability
distributions on $\bbN$. The convergence in the proof above was
componentwise, i.e.~in the product topology, but it is well-known (and
easy to verify) that on $\cP$, this topology is equivalent to the
$\ell^1$-topology with the metric $d((p_i),(p'_i))=\sum_i|p_i-p'_i|$,
and also to the usual weak topology of probability distributions;
hence the theorem holds for any of these topologies.

Let $\tV_i=\gb_i\prod_{j=1}^{i-1}(1-\gb_j)$ be the elements of the
limit sequence in \eqref{tgem}, and let $(V_i)_1^\infty$ denote the
decreasing rearrangements of them.
The distribution of this random element of $\cP$ is denoted
$\PD(\frac1k,\frac1k)$, see 
Pitman and Yor \cite{PitYor1997} or
Bertoin \cite{Bertoin2006}.

Taking the decreasing rearrangement is a continuous operation on
$\cP$, and thus we immediately obtain from \refT{TGEM} the following.

\begin{theorem}\label{TPD}
  \begin{equation}\label{tpd}
	\frac1{kn}(\Kn1,\Kn2,\dots)
\alaw
(V_1,V_2, \dots) \sim \PD\Bigpar{\frac1k,\frac1k}.
  \end{equation}
\end{theorem}

\begin{corollary}
  The largest block size has the limit 
  \begin{equation*}
	\frac{\Kn1}{kn}\alaw V_1=\max_{i\ge1}\tV_i.
  \end{equation*}
\end{corollary}

\begin{remark}
These results can be compared with the classical result that the
lengths of the cycles in a random permutation, arranged in decreasing
order and divided by the size of the permutation, converge (in
distribution) to $\PD(1)=\PD(0,1)$, 
see e.g.~\cite[Sections 5.5 and  5.7]{ArrBarTav2003}.
\end{remark}

\begin{remark}
  For $k=2$ we obtain in \refT{TPD} the limit distribution
  $\PD(\frac12,\frac12)$ which arises in other contexts too: it is the
  distribution of the sequence of excursion lengths in a Brownian 
  bridge \cite{PitYor1992}, \cite{PerPitYor1992}, \cite{AldPit1994},
  \cite{PitYor1997}
(for a related characterization for $k>2$ see \cite{PitYor1997})
and it is the asymptotic distribution of the sizes of the tree components
  in a random mapping, see  \cite{Stepanov1969} and \cite{AldPit1994}.
It is an interesting problem to see whether there are more direct
  relations with these objects.
\end{remark}

\begin{ack}
  Much of this research was done during a visit of SJ to the Erwin
  Schr\"odinger Institute in Vienna, 2008.
\end{ack}

\newcommand\JCTA{\emph{J. Combin. Theory Ser. A} }
\newcommand\book{\emph}
\def\nobibitem#1\par{}

\end{document}